\documentclass[a4paper,11pt,leqno]{amsart}

\usepackage{amssymb,hyperref,latexsym,amsmath}

\theoremstyle{plain}

\newtheorem{theo}{Theorem}[section]

\newtheorem{cor}[theo]{Corollary}

\newtheorem{theorem}[theo]{Theorem}
\newtheorem{lemma}[theo]{Lemma}
\newtheorem{remark}[theo]{Remark}

\theoremstyle{definition}
\newtheorem{defi}[theo]{Definition}

\theoremstyle{remark}

\def\CC{{\mathbb C}}

\def\NN{{\mathbb N}}

\begin{document}

\title[Holomorphic Cartan geometries]{Killing fields of holomorphic Cartan geometries}

\author[S. Dumitrescu]{Sorin DUMITRESCU$^\star$}

\address{${}^\star$ D\'epartement de Math\'ematiques d'Orsay, 
\'equipe de Topologie et Dynamique,
Bat. 425, U.M.R.   8628  C.N.R.S.,
Univ. Paris-Sud (11),
91405 Orsay Cedex, France}
\email{Sorin.Dumitrescu@math.u-psud.fr}
 
 \thanks{This work was  partially supported by  the ANR Grant Symplexe BLAN 06-3-137237}
\date{\today}

\setcounter{tocdepth}{1}

\maketitle

\begin{abstract}    We study local automorphisms of holomorphic Cartan geometries. This leads to  classification results for compact complex manifolds admitting
holomorphic Cartan geometries. We prove that a compact K\"ahler Calabi-Yau manifold bearing a holomorphic Cartan geometry  of  algebraic type admits a finite unramified cover which is a complex torus.\\ 
\end{abstract}

\section{Introduction}

 We study here holomorphic Cartan geometries on complex compact manifolds $M$. 
 
  Let $G$ be a complex connected  Lie group and $P \subset G$ a closed complex Lie subgroup. The Lie algebras of $G$ and $P$ will be denoted by $\mathfrak{g}$ and $\mathfrak p$.

 \begin{defi}  A holomorphic Cartan geometry  $(B, \omega)$  on  $M$   modeled on $G/P$ is a  holomorphic principal right $P$-bundle  $B$ over $M$ endowed with a holomorphic
$\mathfrak g$-valued  one form $\omega$  satisfying:\\
 
 (i)  $\omega_{b} : T_{b}M \to \mathfrak g$ is a linear complex isomorphism  for all $b \in B$.\\

 (ii)  If $X \in \mathfrak  p$ and $X^{\star}$ is the corresponding fundamental vector field on $B$, then $\omega_{b} (X^{\star})=X$, for all $b \in B$.\\

 (iii) $(R_{g})^* \omega =Ad(g^{-1}) \omega$, for all $g \in P$, where $R_{g}$ is the right action on $B$ of $g \in P$.\\
 \end{defi}
 
  If the image of $P$ through  the adjoint representation is  an algebraic subgroup of $Aut(\mathfrak{g})$, the Cartan geometry is said to be  of {\it algebraic type}.

  Recall that a local Killing field  of  the Cartan geometry  is a 
 local holomorphic vector field  on $M$  which  lifts to a vector field on $B$  acting by bundle automorphisms and preserving  $\omega$.  Denote by $Kill^{loc}$ the Lie algebra of local Killing fields (recall that the sheaf of Killing fields on $M$ is locally constant~\cite{Gro} (section 3.5)).
 If $Kill^{loc}$ admits an open orbit $U$ in $M$, we say that $(B, \omega)$ is {\it locally homogeneous} on $U$.

 We show that, in many situations,  a  Cartan geometry  of algebraic type admits  a non trivial algebra $Kill^{loc}$ of  local Killing fields.  This is  inspired by  the celebrated
 stratification  theorem proved by Gromov in the context of {\it rigid geometric structures of algebraic type}~\cite{DG,Gro} (see also the nice Cartan geometries-adapted proof  given by Karin Melnick~\cite{Melnick}). Our result is:

\begin{theorem} \label{result}
Let $M$ be a  compact connected complex manifold of dimension $n$ endowed with  a holomorphic Cartan geometry  $(B,\omega)$ of algebraic type. 
Then :

(i) There exists a compact  analytic subset $S$ of $M$, such that $M \setminus S$ is $Kill^{loc}$-invariant and the orbits of $Kill^{loc}$ in $M \setminus S$
are the fibers of a holomorphic fibration of constant rank.

(ii) For any distinct fibers of the previous fibration there exists  a fibration-invariant meromorphic function on $M$ taking distinct values on them. Consequently,  the dimension of the fibers  is $\geq n-a(M)$, where $a(M)$ is the algebraic dimension of $M$.
\end{theorem}

 \begin{cor} If  $a(M)=0$, then $(B, \omega)$ is locally homogeneous on an open dense set in $M$.
 \end{cor}
 
 \begin{cor} \label{corollaire}  Let $M$ be a compact simply connected complex manifold of positive dimension  with trivial canonical bundle which doesn't admit nonconstant meromorphic functions. Then $M$ doesn't admit holomorphic Cartan geometries of algebraic type.
 \end{cor}
 
 This enables us to prove the following result which was our main motivation:
 
 \begin{theorem}  \label{on Calabi-Yau} A compact K\"ahler  Calabi-Yau manifold  $M$ bearing a  holomorphic Cartan geometry  of algebraic type admits a finite unramified cover which is a complex  torus.
 \end{theorem}

 Benjamin McKay conjectured in~\cite{McKay} that compact K\"ahler  Calabi-Yau manifolds bearing holomorphic Cartan geometries are holomorphically covered by complex tori. Theorem~\ref{on Calabi-Yau}  answers positively  this  conjecture   in  the case of Cartan geometries of algebraic type.  In a recent  paper~\cite{BM}, Indranil Biswas and Benjamin McKay proved the conjecture in the case
 where $M$ is a {\it projective} Calabi-Yau manifold. 
 
 Note that the particular case where $P$ is the complex linear group $GL(n, \CC)$ sitting  in the complex affine group $G= GL(n,\CC) \ltimes \CC^n$ corresponds to a holomorphic affine
 connexion on the holomorphic tangent bundle $TM$. This special case was proved in~\cite{IKO}. The conjecture  was also solved in~\cite{McKay} for  the particular case where $P$ is parabolic or reductive. 
 
 A similar result was proved in~\cite{D1} for holomorphic rigid geometric structures of algebraic affine type  in Gromov's sense~\cite{DG,Gro}, but the context here  is different  since 
  the principal bundle $B$  of a Cartan geometry  is not supposed to be a   frame  bundle of the manifold $M$.
 
We give now the main steps  in   the proof of theorem~\ref{on Calabi-Yau} in   the case where $M$ is a nonprojective Calabi-Yau manifold. If $M$ is  K\"ahler but nonprojective, a result of Moishezon~\cite{Moishezon} implies that
 the algebraic dimension of $M$ is not maximal and theorem~\ref{result}  implies  that any  Cartan geometry on $M$ admits nontrivial local Killing fields. We use then a structure theorem which asserts that, up to a finite cover, $M$  is biholomorphic to a 
 direct product of a simply connected Calabi-Yau manifold with a complex torus~\cite{Beauville} and a
  result of Amores-Nomizu~\cite{Amores, Nomizu}
 about the extendibility of local Killing fields on simply connected manifolds (see also~\cite{DG,Gro,Melnick}).
 
 In the case where $M$ is projective our proof works also for Cartan geometries which are not of algebraic type. It is  a  very simplified version of that given in~\cite{BM}. Actually,  we obtain the
 more general:
 
 \begin{theorem}  \label{produits} Let $N$ be a complex manifold.
 
 (i) If $M$ is a compact simply connected K\"ahler  Calabi-Yau manifold of positive dimension, then $M \times N$ doesn't  admit holomorphic Cartan geometries of algebraic type.
 
 (ii) If $M$ is a compact simply connected projective   Calabi-Yau manifold of positive dimension, then $M \times N$ doesn't  admit holomorphic Cartan geometries.
 
  (iii) If $M$ is a compact simply connected complex  manifold of positive dimension  with  trivial canonical bundle $K_{M}$ (or such that $K_{M}^{-1}$ doesn't admit nontrivial holomorphic sections), then $M \times N$ doesn't  admit flat holomorphic Cartan geometries.
  
  Moreover, the previous three points stand not only for $M \times N$, but also for complex manifolds $W$ containing $M$ as a submanifold such that 
  there exists a holomorphic vector bundle morphism from $TW|_{M}$ onto $TM$.
  \end{theorem}
  
  After finishing  this paper, we learnt from Benjamin McKay that himself and  Indranil Biswas just succeeded to adapt the proof of~\cite{BM} for compact K\"ahler Calabi-Yau manifolds. Their  new unpublished result  uses Bogomolov's $T$-stability theory for coherent sheaves.
 
 \section{Killing fields of Cartan geometries}   \label{context}

 Let $M$ be a complex manifold  endowed with a Cartan geometry of algebraic type modeled on $G/P$.
 
 We can assume without loss  of generality that $P$ contains no nontrivial normal subgroups  of $G$. Indeed, if a  nontrivial  normal subgroup $N$ of $G$
 lies in $P$, then $M$ also admits a Cartan geometry $(B/N, \omega')$  locally modeled on $G'/P'$, where $G'=G/N$ and $P'=P/N$ (see~\cite{Sharpe}, chapter 4).

 Remark that $\omega$ defines a   holomorphic  isomorphism   $TB \simeq B \times \mathfrak  g$, where $TB$ is the holomorphic tangent bundle to $B$.
 
 The curvature  of the Cartan geometry $(B, \omega)$ is a $\mathfrak  g$-valued (holomorphic)  2-form   on $B$  defined by $\Omega(X,Y)=d \omega (X,Y) + \lbrack  \omega(X), \omega (Y) \rbrack,$ for all tangent vector fields $X,Y$ to $B$.  It is well known that $\Omega(X,Y)$ vanishes  if $X \in \mathfrak  p$ (see~\cite{Sharpe}, chapter 5,  corollary 3.10).
 
 The Cartan geometry $(B, \omega)$ is said to be {\it flat}, if the curvature vanish to all of $B$. 
 
 Since  $TB \simeq B \times \mathfrak  g$, the curvature $\Omega$ is completely determined by a  $P$-equivariant function $K : B \to V,$ where $V=\wedge^2(\mathfrak g/ \mathfrak p)^*\otimes \mathfrak g$ and $P$ acts linearly on $V$ by
 
 $$p \cdot l (u,v)=(Ad(p) \circ  l)( \bar {Ad}(p^{-1})u, \bar{Ad}(p^{-1})v),$$
 
 for all $p \in P$, with $\bar{Ad}$ being the induced $P$-action on $\mathfrak{g} / \mathfrak{p}$ coming from the adjoint action $Ad(P)$ (see~\cite{Sharpe}, chapter 5, lemma 3.23).
 
Following~\cite{Melnick} we define for all $m \in \NN$, the $m$-jet of $K$ with respect to $\omega$:

$$J^mK : B \to Hom(\otimes^m \mathfrak{g}, V) $$

$$(J^m K)(b) : X_{1} \otimes X_{2} \otimes \ldots \otimes X_{m} \to (\tilde X_{1}  \cdot \tilde X_{2} \cdot \ldots \tilde X_{m} \cdot K)(b),$$

where $X_{i} \in T_{b}B$ and $\tilde{X}_{i}$ is the unique $\omega$-constant  (holomorphic) vector field on $B$ which extends $X_{i}$.

The $m$-jet of $K$ is $P$-equivariant:

$$J^mK(bp^{-1})= p \cdot (J^mK(b) \circ Ad^mp^{-1}),$$
where $Ad^m$ is the tensor representation $\otimes^m \mathfrak g$ of $AdP$ (see proposition 3.1 in~\cite{Melnick}).

\begin{defi} A (local) automorphism of a Cartan geometry $(B, \omega)$ on $M$ is a (local) biholomorphism $f$ of $M$ which lifts to a (local) bundle automorphism
of  $B$ preserving $\omega$.
\end{defi}

Conversely, a local bundle automorphism of $B$ preserving $\omega$ is the lift of a unique  local biholomorphism of $M$~\cite{Melnick} (proposition 3.6).

The pseudogroup of local automorphisms of a Cartan geometry is a Lie pseudogroup $Is^{loc}$  (of finite dimension) generated by the  Killing Lie algebra of the Cartan geometry $Kill^{loc}$.  We will say that $m,n \in M$ are in the same $Kill^{loc}$-orbit if $n$ can be reached from $m$ by flowing along a finite sequence of local
Killing fields. Locally the orbits of $Is^{loc}$ and the orbits of $Kill^{loc}$ are the same.

The following  well known lemma will be useful in the sequel (see also~\cite{Sharpe}, Appendix A). 

\begin{lemma} \label{connection} Let $M$ be a complex manifold endowed with a holomorphic Cartan geometry $(B, \omega)$. The holomorphic $\mathfrak{g}$-bundle $B_{\mathfrak{g}}$ over  $M$ corresponding   to $B$ by the adjoint action of the structure group $P$ on $\mathfrak{g}$ admits a holomorphic affine connection.
\end{lemma}

\begin{proof}
 It is enough  to prove  that the holomorphic principal $G$-bundle $B_{G}=B \times_{P} G$ obtained by extending the structure group of $B$ using the inclusion map of $P$ in $G$ admits a holomorphic  connection~\cite{Kobayashi}.
 
 Consider $\omega_{MC} :TG \to G \times \mathfrak{g}$ be the $\mathfrak g$-valued Maurer-Cartan one form on $G$ constructed using the left invariant vector fields. Consider the $\mathfrak{g}$-valued holomorphic
 one form $$\tilde \omega(b,g)=Ad(g^{-1}) \pi_{1}^*\omega + \pi_{2}^* \omega_{MC},$$
 on $B \times G$, where the $\pi_{i}$ are the canonical projections on the two factors. The one form $\tilde \omega$ descends on $B_{G}$ to a $\mathfrak{g}$-valued
 holomorphic one form which defines a holomorphic connection.
 \end{proof}

\begin{remark}  \label{hyp} (i) The proof of Lemma~\ref{connection}  only requires   points  (ii)  and (iii) in the definition of a  Cartan geometry.

(ii) If the Cartan geometry $(B, \omega)$ is flat, then the holomorphic affine connection constructed in Lemma~\ref{connection} is easily seen to be flat.
\end{remark}

 \section{Cartan geometries and algebraic dimension}
 
  The maximal number of algebraically  independent meromorphic functions on a complex manifold $M$ is called {\it the algebraic dimension} $a(M)$ of $M$. 
     
     Recall that a theorem of Siegel proves that a complex $n$-manifold $M$ admits at most $n$ algebraically  independent  meromorphic functions~\cite{Ueno}.
      Then $a(M) \in \{0, 1, \ldots, n\}$ and for algebraic manifolds $a(M)=n$.

   We will say that two points in  $M$ are in the same {\it fiber of the algebraic reduction} of $M$ if any meromorphic  function on $M$ takes the same value at  the two points. There
   exists some open dense set in  $M$ where the fibers of the algebraic  reduction are the fibers of a holomorphic fibration  on an algebraic manifold of dimension $a(M)$ and any 
   meromorphic function on $M$ is the pull-back of a meromorphic function on the basis~\cite{Ueno}.

   Theorem~\ref{result} shows  that the fibers of the algebraic reduction are in the same orbit of the pseudogroup of local isometries for any holomorphic Cartan geometry  of algebraic type  on $M$. 
   Let's give the proof.

\begin{proof}
(i) For each positive  integer  $m$  we consider the $m$-jet  $J^mK$ of  the curvature of  $(B, \omega)$. This is a $P$-equivariant holomorphic  map 
    $$J^mK : P  \to W= Hom(\otimes^m \mathfrak{g}, V).$$
    
    The proof of theorem 4.1 in~\cite{Melnick} shows that two points in $M$ are in the same $Kill^{loc}$-orbit  if and only if the corresponding fibers of $B$ are sent on the same $P$-orbit in $W= Hom(\otimes^m \mathfrak{g}, V)$, for a certain  $m$ large enough.
   
    Since the $Ad(P)$-action on $W$ is supposed to be algebraic,  Rosenlicht's  theorem (see~\cite{Ro}) shows that there exists  a $P$-invariant stratification 
     
     $$W=Z_0 \supset \ldots \supset Z_l,$$
      such that 
     $Z_{i+1}$ is Zariski closed in  $Z_i$, the quotient of $Z_{i} \setminus Z_{i+1}$ by $P$ is a complex manifold  and 
     rational $P$-invariant functions  on  $Z_i$ separate
      orbits in  $Z_i \setminus Z_{i+1}$.
     
     Consider  the open dense $Kill^{loc}$-invariant subset $U$ of $M$, where $J^mK$ is of constant rank and the 
      image of  $B|_{U}$ through  $J^mK$ is contained in the  maximal subset  $Z_i \setminus Z_{i+1}$ of the stratification which intersects the image of $J^mK$. Then $U=M \setminus S$, with $S$ a compact analytic subset in $M$, and 
      the orbits of $Kill^{loc}$ in $U$
      are the fibers of a fibration of constant rank.
      
 (ii)      If $n$ and $n'$ are two points in $U$ which are not in the same $Kill^{loc}$-orbit, then the corresponding  fibers of $B|_{U}$ are sent by $J^mK$  on two
      distinct $P$-orbits in $Z_i \setminus Z_{i+1}$. By Rosenlicht's theorem there exists a $P$-invariant  rational function $F : Z_i \setminus Z_{i+1} \to \CC,$
      which takes distincts values at these two orbits.

     The meromorphic function  $F \circ J^mK : B  \to \CC$ is $P$-invariant and descends in a $Kill^{loc}$-invariant meromorphic function on 
     $M$ which takes distincts values at  $n$ and at $n'$.
     
    Consequently, the complex codimension in $U$ of the $Kill^{loc}$-orbits is $\leq a(M)$, which finishes the proof.
    \end{proof}

 \section{Cartan geometries and  simply-connected manifolds}

We prove first the corollary~\ref{corollaire}.

\begin{proof}

Assume, by contradiction, that  the complex manifold $M$ bearing the Cartan geometry  $(B, \omega)$ verifies   the hypothesis. Since $a(M)=0$, theorem~\ref{result} implies  $(B, \omega)$ is locally homogeneous on an open dense set $U$ in $M$. As $M$
is simply connected, elements  in  the Killing Lie algebra $\mathcal G$ extend to all of  $M$~\cite{Amores,Gro, Melnick, Nomizu}: the unique  connected simply connected complex Lie group $G'$ associated to $\mathcal G$  acts isometrically on $M$ with an open dense orbit. The open dense orbit $U$  identifies with a homogeneous space $G'/H$, where $H$ is a closed subgroup of $G'$.

Consider  $X_{1}, X_{2}, \ldots, X_{n}$  global Killing fields on $M$ which are linearly independent at some point of the open orbit $U$. Consider the function $vol(X_{1},X_{2}, \ldots, X_{n})$, where  $vol$ is the holomorphic volume form associated to a nontrivial section of the canonical bundle. Since $vol(X_{1}, X_{2}, \ldots, X_{n})$ is a holomorphic function on $M$, it
is a nonzero  constant (by maximum principle) and, consequently, $X_{1}, X_{2}, \ldots,  X_{n}$ are linearly independent on $M$. Hence Wang's theorem~\cite{Wang}  implies  that $M$ is a quotient of a $n$-dimensional
connected simply connected complex Lie group $G_{1}$ by a discrete subgroup. Since $M$ is simply connected, this discrete subgroup has to be trivial and $M$ identifies with
$G_{1}$. But there is no compact  simply connected complex Lie group: a contradiction.
\end{proof}

\begin{theorem} \label{dim3}  Let $M$ be a compact connected simply connected complex $n$-manifold without nonconstant meromorphic functions and admitting a holomorphic Cartan geometry  $(B,\omega)$ of algebraic type. Then $M$ is biholomorphic to an equivariant compactification of $\Gamma \backslash G'$, where $\Gamma$ is a discrete noncocompact subgroup in  a complex Lie group $G'$.
\end{theorem}

\begin{proof}
Since $a(M)=0$, theorem~\ref{result} implies  $(B, \omega)$ is locally homogeneous on an open dense set $U$. As before, the extension property of local Killing fields implies $U$ is 
a complex homogeneous space $G'/H$, where $G'$ is a connected simply connected complex Lie group and $H$ is a closed subgroup  in $G'$.

We show now that $H$ is a discrete subgroup of $G'$. Assume by contradiction the Lie algebra of $H$  is nontrivial.
Take at any point $u \in U$, the isotropy subalgebra $\mathcal{H}_{u}$ (i.e. the Lie subalgebra of  Killing fields vanishing at $u$).
Remark that $\mathcal{H}_{gu}=Ad(g) \mathcal{H}_{u}$, for any $g \in G'$ and $u \in U$, where $Ad$ is the adjoint representation. Let $d$ be the complex dimension of $\mathcal{H}_{u}$.

The map $u \to \mathcal{H}_{u}$ is a meromorphic map from $M$ to the grassmanian of $d$-dimensional vector spaces in $\mathcal G$. Since $M$ doesn't admit  nontrivial 
meromorphic function, this map has to be constant. It follows that $\mathcal {H}_{u}$ is $Ad(G')$-invariant and $H$ is a normal subgroup of $G'$: a contradiction, since the $G'$-action
on $M$ is faithful. Thus $G'$ is of dimension $n$ and $H$ is  a discrete subgroup  $\Gamma$ in $G'$.

As $M$ is simply connected,  $U$ has to be strictly contained  in $M$ and $M$ is an equivariant compactification of $\Gamma  \backslash G'$.
\end{proof}

We don't know if such compactifications of $ \Gamma \backslash G'$ admit  holomorphic Cartan geometries, but the previous result has the following application.

 Recall that  an open question asks  whether  the $6$-dimensional real sphere $S^6$ admits  complex structures or no. In this context, we have the following:

\begin{cor} If $S^6$ admits a complex structure $M$, then $M$  doesn't admit holomorphic Cartan geometries of algebraic type.
\end{cor}

\begin{proof} The starting point of the proof is a result of~\cite{CDP} where it is proved that $M$ doesn't admit  nonconstant meromorphic functions. If $M$
admits a   holomorphic Cartan geometry, then the previous proof shows that $M$ supports a holomorphic action of a three-dimensional complex Lie group $G'$
with an open orbit. This is in contradiction with the main theorem  of~\cite{HKP}.
\end{proof}

\section{Cartan geometries and Calabi-Yau manifolds}

Recall that K\"ahler Calabi-Yau manifolds are K\"ahler manifolds with vanishing first (real) Chern class~\cite{Beauville}.

The aim of this section is to prove theorem~\ref{on Calabi-Yau}. We settle first the case where $M$ is simply connected.

\begin{lemma}  \label{CY} A simply connected K\"ahler  Calabi-Yau manifold of positive dimension $n$   doesn't admit holomorphic Cartan geometries of algebraic type.
\end{lemma}

 \begin{proof} 
Assume first $M$ is nonprojective.  A  theorem  of  Moishezon~\cite{Moishezon}  shows  that the algebraic dimension of a K\"ahler  nonprojective complex manifold of dimension $n$  is $\leq n-1$. Theorem~\ref{result} implies then that the Killing
 Lie algebra of a Cartan geometry on $M$ is nontrivial. Since $M$ is simply-connected, a  nontrivial element of the Killing Lie algebra extends to a global holomorphic (Killing) vector field on $M$~\cite{Amores, Gro, Melnick, Nomizu}.
 
 But a simply connected compact Calabi-Yau manifold doesn't admit nontrivial holomorphic vector fields~\cite{Kobayashi-Horst}: a contradiction.
 
 Consider now the case where $M$ is projective.  Assume by contradiction that $M$ admits a Cartan geometry   $(B, \omega)$  locally modelled on $G/P$.

 Let $B_{\mathfrak{g}}$ (respectively $B_{\mathfrak{p}}$) be the holomorphic vector bundle over $M$  with fiber $\mathfrak{g}$ (respectively $\mathfrak{p}$) associated to $B$ and corresponding to  the action of the structure group $P$ on $\mathfrak{g}$ (respectively on $\mathfrak{p}$) by the adjoint representation.
 
The  point (i) in the definition of a Cartan geometry implies that the holomorphic tangent bundle $TM$ is isomorphic to $B_{\mathfrak{g}} / B_{\mathfrak{p}}$ which is also the holomorphic vector bundle $B_{\mathfrak{g / p}}$ corresponding to the adjoint $P$-action  on the quotient $\mathfrak{g} / \mathfrak{p}$.

Lemma~\ref{connection}    shows  that $B_{\mathfrak{g}}$ admits a holomorphic affine connection. By theorem A(1) in~\cite{Biswas}, $B_{\mathfrak{g}}$ also admits a {\it flat} holomorphic affine connection. Since $M$ is simply connected, $B_{\mathfrak{g}}$ is holomorphically  trivial . 

 Let $p$ be the canonical projection of $B_{\mathfrak{g}}$
onto $TM$. Choose   $s_{1}, \ldots, s_{n}$ global  holomorphic sections of  $B_{\mathfrak{g}}$ such that $p(s_{1}), \ldots, p(s_{n})$ span $TM$ at a chosen point in $M$.
Then $p(s_{1}) \wedge \ldots \wedge p(s_{n})$ is a nontrivial holomorphic section of $K_{M}^{-1}$, where $K_{M}$ is the canonical bundle of $M$. Since $K_{M}$ is trivial,
the section $p(s_{1}) \wedge \ldots \wedge p(s_{n})$ doesn't vanish on $M$. Consequently, $p(s_{1}), \ldots, p(s_{n})$ trivialize $TM$. Wang's theorem~\cite{Wang} implies that $M$
is a complex torus: a contradiction (for $M$ is supposed to be simply connected).
 \end{proof}
 
 We give now the proof of theorem~\ref{on Calabi-Yau}.
 
 \begin{proof} Let $M$ be a K\"ahler  Calabi-Yau manifold bearing a holomorphic Cartan geometry  $(B, \omega)$. It is known that, up to  a finite unramified
 cover,  $M$ is biholomorphic to a product of nontrivial simply connected K\"ahler Calabi-Yau manifolds $M_{1}, M_{2}, \ldots M_{l}$ and a complex torus $\CC^p / \Lambda$, with $\Lambda$ being a lattice in $\CC^p$~\cite{Beauville}.
 
 Remark that the proof of lemma~\ref{CY} still  works if  at least one of the simply connected factors $M_{i}$ is nonprojective. Indeed, in this case meromorphic  functions on $M$ don't separate points
 in the  fibers  $M_{1} \times M_{2} \times \ldots \times M_{l} \times \{t \}$, with $t \in \CC^p / \Lambda$,  and the proof of theorem~\ref{result} shows that the foliation given by the  $Kill^{loc}$-orbits  intersects a generic fiber, say  $M_{1} \times M_{2} \times \ldots \times M_{l} \times \{t \}$,
 on a  foliation with positive dimensional  leaves. Hence  we can choose a  local Killing field  $X$ defined on a connected simply connected  open set $U$ in $M$ which is tangent to $M_{1} \times M_{2} \times \ldots \times M_{l} \times \{t \}$ at a given  point 
 $m  \in U \cap (M_{1} \times M_{2} \times \ldots \times M_{l} \times \{t \})$ and $X(m) \neq 0$.

 By Amores-Nomizu's extendibility result, $X$ extends to  a holomorphic Killing vector field $\tilde X$ on $M_{1} \times M_{2} \times \ldots \times M_{l} \times U'$, with $U'$ being a simply connected open set in $\CC^p /  \Lambda$ containing $t$. Consider 
 the image of $\tilde X$ through the projection on the first factor of the canonical decomposition $$T(M_{1} \times M_{2} \times \ldots \times M_{l} \times U') \simeq \pi_{1}^*T(M_{1} \times M_{2} \times \ldots \times M_{l}) \oplus \pi_{2}^*TU'$$ where the $\pi_{i}$ are the canonical projections on the simply connected factor and on the torus.
 
 We constructed  a  holomorphic vector field on the simply connected Calabi-Yau manifold  $M_{1} \times M_{2} \times \ldots \times M_{l} \times \{t \}$~ which doesn't vanish at $m$. This is in contradiction with~\cite{Kobayashi-Horst} as before.
 
 Consider now the remaining  case where all simply connected factors $M_{i}$ are projective. Assume, by contradiction,  that $M_{1} \times M_{2} \times \ldots \times M_{l}$ is nontrivial.
 
 Let $B_{\mathfrak{g}}$ (respectively $B_{\mathfrak{p}}$) be the holomorphic vector bundle over $M$  with fiber $\mathfrak{g}$ (respectively $\mathfrak{p}$) associated to $B$ and corresponding to  the action of the structure group $P$ on $\mathfrak{g}$ (respectively on $\mathfrak{p}$) by the adjoint representation. As before $TM$ is isomorphic to $B_{\mathfrak{g}} / B_{\mathfrak{p}}$. Let $p$ be the canonical projection of $B_{\mathfrak{g}}$ onto $TM$.
 
 Let $p_1$ (respectively $p_2$) be the projections on the first factor (respectively second factor) of the decomposition $$TM = \pi_{1}^*T(M_{1} \times M_{2}  \times \ldots \times M_{l}) \oplus \pi_{2}^*T(\CC^p /  \Lambda).$$
 
 Then $p_1 \circ p : B_{\mathfrak{g}} \to  \pi_{1}^*T(M_{1} \times M_{2}  \times \ldots \times M_{l})$ is a surjective     morphism of vector bundles. When restricted to a fiber $M_{1} \times M_{2} \times \ldots \times M_{l} \times \{t \}$, this induces a surjective morphism  $\bar{B}_{\mathfrak{g}} \to T(M_{1} \times M_{2}  \times \ldots \times M_{l})$, where
 $\bar{B}_{\mathfrak{g}}$  denotes the restriction of $B_{\mathfrak{g}}$ to $M_{1} \times M_{2} \times \ldots \times M_{l} \times \{t \}$.

 On the other hand, Lemma~\ref{connection}    shows   that  $B_{\mathfrak{g}}$ admits a holomorphic affine connection.  In particular,  $\bar{B}_{\mathfrak{g}}$ admits a holomorphic affine connection. By theorem A(1) in~\cite{Biswas}, $\bar{B}_{\mathfrak{g}}$ also admits a {\it flat} holomorphic affine connection. Since $M_{1} \times M_{2}  \times \ldots \times M_{l}$ is simply connected, $\bar{B}_{\mathfrak{g}}$ is holomorphically trivial. We obtained that $T(M_{1} \times M_{2}  \times \ldots \times M_{l})$ is a quotient of a  holomorphically trivial vector bundle. We conclude as in the proof of Lemma~\ref{CY}.
\end{proof}

Remark that, in the case where  the simply connected factor $M_{1} \times M_{2}  \times \ldots \times M_{l}$ is projective, the previous proof also  works for Cartan geometries which are not  necessarily of algebraic type.

Actually the previous proof doesn't use neither  that  the second factor of the holomorphic decomposition of a  Calabi-Yau manifold as a product is a complex torus.  In our proof this  second factor might be
any complex manifold $N$. This proves also points (i) and (ii) in theorem~\ref{produits}. Whereas  for point (iii) of theorem~\ref{produits} we give the proof here:

\begin{proof}   Consider a flat  Cartan geometry $(B, \omega)$ on $M \times N$ and the associated  vector bundle $B_{\mathfrak{g}}$ over $M \times N$.
As before  $TM$ is a quotient of  the vector bundle $\bar{B}_{\mathfrak{g}}$, where $\bar{B}_{\mathfrak{g}}$ denotes the restriction to the factor $M$ of $B_{\mathfrak{g}}$.

Since the Cartan geometry $(B, \omega)$ is supposed to be flat, Lemma~\ref{connection}  and Remark~\ref{hyp} (point (ii))  show  that $B_{\mathfrak{g}}$ and hence also $\bar{B}_{\mathfrak{g}}$ admit a flat  holomorphic affine connection. 
As $M$ is simply connected, the vector bundle $\bar{B}_{\mathfrak{g}}$ over $M$  is holomorphically trivial. 

As in the proof of Lemma~\ref{CY}, let $p$ be the canonical projection of $\bar{B}_{\mathfrak{g}}$
onto $TM$. Let $n$ be the complex dimension of $M$ and choose   $s_{1}, \ldots, s_{n}$ global  holomorphic sections of  $\bar{B}_{\mathfrak{g}}$ such that $p(s_{1}), \ldots, p(s_{n})$ span $TM$ at a chosen point in $M$.
Then $p(s_{1}) \wedge \ldots \wedge p(s_{n})$ is a nontrivial holomorphic section of $K_{M}^{-1}$, where $K_{M}$ is the canonical bundle of $M$. We get a contradiction in the case where 
$K_{M}^{-1}$ doesn't admit nontrivial holomorphic sections.

Consider now the case where  $K_{M}$ is trivial. Then
the section $p(s_{1}) \wedge \ldots \wedge p(s_{n})$ doesn't vanish on $M$. Consequently, $p(s_{1}), \ldots, p(s_{n})$ trivialize $TM$. Wang's theorem~\cite{Wang} implies that $M$
is a quotient of a nontrivial
connected simply connected complex Lie group $G_{1}$ by a discrete subgroup. Since $M$ is simply connected, this discrete subgroup has to be trivial and $M$ identifies with
$G_{1}$. But there is no nontrivial compact  simply connected complex Lie group: a contradiction.

Remark that all over the proof we didn't use the product structure of the manifold $W=M \times N$, but only the existence of a holomorphic vector bundle morphism from $TW|_{M}$ onto $TM$.
\end{proof}

\end{document}